\documentclass{amsart}
\usepackage{amssymb, amsmath, amsthm, graphics, comment, xspace, enumerate}
\newtheorem{thm}{Theorem}[section]
\newtheorem{cor}[thm]{Corollary}
\newtheorem{lem}[thm]{Lemma}
\newtheorem{prop}[thm]{Proposition}
\theoremstyle{definition}

\theoremstyle{remark}
\newtheorem{rem}[thm]{Remark}
\numberwithin{equation}{section}

\begin{document}
\title[An Efficient Abstract Method For The Study...]{An Efficient Abstract Method For The Study of An Initial Boundary
Value Problem On Singular Domain}

\author{Belkacem Chaouchi}
\address{Lab. de l'Energie et des Syst\' emes Intelligents, Khemis Miliana University, 44225 Khemis Miliana, Algeria}
\email{chaouchicukm@gmail.com}

\author{Marko Kosti\' c}
\address{Faculty of Technical Sciences,
University of Novi Sad,
Trg D. Obradovi\' ca 6, 21125 Novi Sad, Serbia}
\email{marco.s@verat.net}

{\renewcommand{\thefootnote}{} \footnote{
\\ \text{  }  \ \    {\it Key words and phrases.} Abstract differential equations of second
order, variable operator coefficients, anti-periodic boundary conditions, H\"{o}lder
spaces.
\\  \text{  }  \ \ The second named author is partially supported by grant 174024 of Ministry
of Science and Technological Development, Republic of Serbia.}}

\begin{abstract}
The present work is devoted to the study of a boundary value problem for
second order linear differential equation set on singular cylindrical
domain. This problem can be regarded via a natural change of variables as an
elliptic abstract differential equation with variable operators coefficients
subject to some anti-periodic conditions. The complete study of this
abstract version allows us to establish some interesting regularity results
for our problem. The study is performed in the framework of H\"{o}lder
spaces.
\end{abstract}

\maketitle

\section{Introduction and preliminaries}

In \cite{chaouchimjm} and \cite{chaouchi 2}, the solvability of some BVP's
set on cusp domain was discussed. The authors opted for the use of the
abstract differential equations theory and some regularity results for these
problems were successfully established in the framework of Little H\"{o}lder
spaces. In the same direction, we will show that this approach can be
exploited in order to give a complete study of an initial boundary value
problem involving the Laplace operator and which is also posed on nonsmooth
domain. More precisely, we consider a particular conical domain $\Omega $
given by 
\begin{equation*}
\Pi =\left[ 0,T\right] \times \Omega ,
\end{equation*}%
where%
\begin{equation*}
\Omega =\left \{ \left( x,y\right) \in 
\mathbb{R}
^{2}:\sqrt{x^{2}+y^{2}}\leq \varphi \left( t\right) \right \} .
\end{equation*}%
Here, $\varphi $ is a positive real-valued function of parametrization
defined on $\left[ 0,1\right] $ such that%
\begin{equation*}
\begin{array}{c}
\varphi \left( 0\right) =\varphi ^{\prime }\left( 0\right) =0.%
\end{array}%
\end{equation*}%
In $\Pi ,$ we consider the following problem 
\begin{equation}
\begin{array}{c}
\partial _{t}^{2}u+\Delta u-\lambda u=h,\text{ }\lambda >0,%
\end{array}
\label{wave operator}
\end{equation}%
corresponding to the following initial conditions 
\begin{equation}
\begin{array}{lll}
\left. u\right \vert _{t=0}+\left. u\right \vert _{D\left( T,\varphi \left(
T\right) \right) }=0, &  & \left. \partial _{t}u\right \vert _{t=0}+\left.
\partial _{t}u\right \vert _{D\left( T,\varphi \left( T\right) \right) }=0,%
\end{array}
\label{initial conditions}
\end{equation}%
where $D\left( T,\varphi \left( T\right) \right) $ denotes the disc of
radius $\varphi \left( T\right) $ centred at $(T,0,0).$

We accompany (\ref{wave operator})-(\ref{initial conditions}) with following
additional conditions%
\begin{equation}
\begin{array}{ccc}
\left. u\right \vert _{\partial \Pi \backslash D\left( T,\varphi \left(
T\right) \right) }=0, &  & \left. u\right \vert _{\partial \Pi \backslash
\left \{ 0\right \} }=0.%
\end{array}
\label{boundary conditions}
\end{equation}%
We assume that the right hand term of (\ref{wave operator}) is taken in the
anistropic H\"{o}lder space $C^{2\theta }\left( \left[ 0,T\right] ;C\left(
\Omega \right) \right) ,0<2\theta <1$, defined by 
\begin{equation*}
\left \{ \phi \in C(\left[ 0,T\right] ;C(\Omega )):\lim \limits_{\varepsilon
\rightarrow 0^{+}}\sup \limits_{0<\left \vert t-t^{\prime }\right \vert \leq
\varepsilon }\dfrac{\left \Vert \phi (t)-\phi (t^{\prime })\right \Vert
_{C(\Omega )}}{\left \vert t-t^{\prime }\right \vert ^{2\theta }}<\infty
\right \} .
\end{equation*}%
Now, we consider the following change of variables 
\begin{equation}
\begin{array}{ll}
T: & \Pi \rightarrow Q, \\ 
& \left( t,x,y\right) \mapsto \left( t,\xi ,\eta \right) =\left( t,\dfrac{x}{%
\varphi \left( t\right) },\dfrac{y}{\varphi \left( t\right) }\right) ,%
\end{array}
\label{change of variables}
\end{equation}%
where%
\begin{equation*}
Q=\left[ 0,T\right] \times D,
\end{equation*}%
with%
\begin{equation*}
D:=D\left( 0,1\right) =\left \{ \left( \xi ,\eta \right) \in 
\mathbb{R}
^{2}:\xi ^{2}+\eta ^{2}\leq 1\right \} .
\end{equation*}%
Define the following change of functions%
\begin{equation}
\begin{array}{l}
u\left( t,x,y\right) =v\left( t,\xi ,\eta \right) \text{ and }h\left(
t,x,y\right) =f\left( t,\xi ,\eta \right) .%
\end{array}
\label{change of functions}
\end{equation}%
It follows from the change of functions (\ref{change of functions}) that the
new version of problem (\ref{wave operator}) is given by 
\begin{equation}
\begin{array}{ll}
\partial _{t}^{2}v+L\left( t\right) v-\lambda v=f, & \text{in }Q \\ 
\left. v\right \vert _{\left \{ 0\right \} \times D}+\left. v\right \vert
_{\left \{ T\right \} \times D}=0, &  \\ 
\left. \partial _{t}v\right \vert _{\left \{ 0\right \} \times D}+\left.
\partial _{t}v\right \vert _{\left \{ T\right \} \times D}=0, &  \\ 
\left. v\right \vert _{\left[ 0,T\right] \times \partial D}=0, & 
\end{array}
\label{transformed problem}
\end{equation}%
where $L$ is the linear operator with singular coefficients given by%
\begin{equation*}
L\left( t\right) =\frac{1}{\varphi ^{2}\left( t\right) }\Delta +\frac{%
\varphi ^{\prime }\left( t\right) }{\varphi \left( t\right) }\left \{ \xi
\partial _{\xi }+\eta \partial _{\eta }\right \} ,\text{ }0\leq t\leq T.
\end{equation*}%
The following lemma is needed in order to clarify the impact of the change
of variables (\ref{change of variables}) on the functional framework of H%
\"{o}lder Spaces.

\begin{lem}
\label{effect of change of variables}\textit{Let} $0<2\theta <1$\textit{.
Then}
\end{lem}

\begin{enumerate}
\item $h\in C^{2\theta }(\left[ 0,T\right] ;C(\Omega ))\Rightarrow f\in
C^{2\theta }(\left[ 0,T\right] ;C(D))$\textit{.}

\item $f\in C^{2\theta }(\left[ 0,T\right] ;C(D))\Rightarrow h\in
C_{w}^{2\theta }(\left[ 0,T\right] ;C(\Omega ))$ with 
\begin{equation*}
C_{w}^{2\theta }(\left[ 0,T\right] ;C(\Omega )))=\left \{ h\in C^{2\theta }(%
\left[ 0,T\right] ;C(\Omega )) : \left( \varphi \left(
.\right) \right) ^{2\theta }h\in C^{2\theta }(\left[ 0,T\right] ;C(\Omega
))\right \} .
\end{equation*}
\end{enumerate}

\begin{proof}
See Proposition 3.1 in \cite{chaouchimjm}.
\end{proof}

Due to the presence of a singular coefficients, we must approximate
the cylinder $\Pi $ by a sequence of regular subdomains. As in \cite%
{sadallah}, we perform the following regular change of variables given by 
\begin{equation*}
\begin{array}{l}
\Pi _{n}:=\left[ t_{n},T\right] \times \Omega \rightarrow Q_{n}:=\left[
t_{n},T\right] \times D\left( 0,1\right) , \\ 
\left( t,x,y\right) \mapsto \left( t,\xi ,\eta \right) =\left( t,\dfrac{x}{%
\varphi \left( t\right) },\dfrac{y}{\varphi \left( t\right) }\right) .%
\end{array}%
\end{equation*}%
Here, $\left( t_{n}\right) _{n\in 
\mathbb{N}
}$ is a decreasing sequence such that $0\leq t_{n}\leq 1$ and $%
\lim \limits_{n\rightarrow +\infty }t_{n}=0.$

Set%
\begin{equation*}
\left \{ 
\begin{array}{l}
v_{n}=\left. v\right \vert _{Q_{n}}, \\ 
\\ 
f_{n}=\left. f\right \vert _{Q_{n}}.%
\end{array}%
\right. 
\end{equation*}%
Summing up, we are confronted to the study of following problems 
\begin{equation}
\begin{array}{lc}
\partial _{t}^{2}v_{n}+L\left( t\right) v_{n}-\lambda v_{n}=f_{n}, & \text{
in }Q_{n}, \\ 
\left. v_{n}\right \vert _{\left \{ t_{n}\right \} \times D}+\left.
v_{n}\right \vert _{\left \{ T\right \} \times D}=0, &  \\ 
\left. \partial _{t}v_{n}\right \vert _{\left \{ t_{n}\right \} \times
D}+\left. \partial _{t}v_{n}\right \vert _{\left \{ T\right \} \times D}=0, & 
\\ 
\left. v_{n}\right \vert _{\left[ t_{n},T\right] \times \partial D}=0. & 
\end{array}
\label{complete transformed problem inQn}
\end{equation}%
Here, we just briefly note that 
\begin{equation*}
\lim \limits_{n\rightarrow +\infty }v_{n}=\left. v\right \vert
_{\lim \limits_{n\rightarrow +\infty }Q_{n}}=\left. v\right \vert _{Q}=v.
\end{equation*}%
In the next section, we will show that our transformed problems (\ref%
{complete transformed problem inQn}) can be formulated as a second order
abstract differential equation of elliptic type with variable operators
coefficients.

\section{The abstract formulation of the problems (\protect \ref{complete
transformed problem inQn})}

Let us introduce the following vector-valued functions:%
\begin{eqnarray*}
v_{n} &:&\left[ t_{n},T\right] \rightarrow E;\text{ }t\longrightarrow
v_{n}(t);\text{\quad }v_{n}(t)(\xi ,\eta )=v_{n}\left( t,\xi ,\eta \right) ,
\\
f_{n} &:&\left[ t_{n},T\right] \rightarrow E;\text{ }t\longrightarrow
f_{n}(t);\quad f_{n}(t)(\xi ,\eta )=f_{n}\left( t,\xi ,\eta \right) ,
\end{eqnarray*}%
with $E=C(D)$. So, the transformed problem (\ref{complete transformed
problem inQn}) can be formulated as follows 
\begin{equation}
\begin{array}{ll}
v_{n}^{\prime \prime }(t)+A\left( t\right) v_{n}(t)-\lambda
v_{n}(t)=f_{n}(t), & t_{n}\leq t\leq T, \\ 
v_{n}(t_{n})+v_{n}(T)=0, &  \\ 
v_{n}^{\prime }(t_{n})+v_{n}^{\prime }(T)=0. & 
\end{array}
\label{approached ADE in Qn}
\end{equation}%
Here, $\left( A\left( t\right) \right) _{t_{n}\leq t\leq T}$ is a family of
closed linear operators with domains $D(A(t))$ (which are not dense) defined by
\begin{equation}
\left \{ 
\begin{array}{l}
D\left( A\left( t\right) \right) :=\left \{ \phi \in W^{2,p}(D)\cap C_{0}(D),%
\text{ }p>2:L\left( t\right) \phi \in C_{0}(D)\right \} ,\text{ }t_{n}\leq
t\leq T, \\ 
\left( A\left( t\right) \right) \phi \left( \xi ,\eta \right) :=\left(
L\left( t\right) \right) \phi \left( \xi ,\eta \right) .%
\end{array}%
\right.  \label{Operator Ln}
\end{equation}%
Consider the natural change of function%
\begin{equation*}
w_{n}(t)=v_{n}(t+t_{n})\text{ and\ }g_{n}(t)=f_{n}(t+t_{n});
\end{equation*}%
then 
\begin{equation}
g_{n}\in C^{2\theta }(\left[ 0,T\right] ;C(D));  \label{RegularityOfF}
\end{equation}%
and $w_{n}$ is the eventual solution of 
\begin{equation}
\left \{ 
\begin{array}{l}
w_{n}^{\prime \prime }(t)+A(t+t_{n})w_{n}(t)-\lambda w_{n}(t)=g_{n}(t),\text{%
\ }0\leq t\leq T, \\ 
w_{n}(0)+w_{n}(T)=0, \\ 
w_{n}^{\prime }(0)+w_{n}^{\prime }(T)=0,%
\end{array}%
\right.
\end{equation}

From \cite{Aquistapace-Terreni} p. 60, we know that the family $\left(
A(t+t_{n})\right) _{0\leq t\leq T}$ enjoys the following three properties:

\begin{enumerate}
\item 
\begin{equation}
\exists M>0,\text{ }\forall z\geqslant 0,\forall t\in \lbrack 0,T]\text{ \ }%
\left \Vert (A_{n}(t+t_{n})-z)^{-1}\right \Vert _{L(E)}\leqslant \dfrac{M}{z+1}%
;  \label{Hypoth1}
\end{equation}

\item \textit{For all }$z\geqslant 0$\textit{, the application }$t\mapsto
(A_{n}(t+t_{n})-\lambda -z)^{-1}$\textit{\ defined on }$\left[ 0,T\right] $%
\textit{\ is in }$C^{2}(\left[ 0,T\right] ;L(E))$\textit{\ and there exist }$%
C>0$\textit{\ such that :}%
\begin{equation}
\forall z\geqslant 0,\forall t\in \lbrack 0,T]\  \text{\ }\left \Vert \frac{%
\partial }{\partial t}(A(t+t_{n})-\lambda -zI)^{-1}\right \Vert
_{L(E)}\leqslant \dfrac{C}{z+1},  \label{Hypoth2}
\end{equation}%
\textit{and}%
\begin{equation}
\forall z\geqslant 0,\forall t\in \lbrack 0,T]\  \text{\ }\left \Vert \frac{%
\partial ^{2}}{\partial t^{2}}(A(t+t_{n})-\lambda -zI)^{-1}\right \Vert
_{L(E)}\leqslant \dfrac{C}{z+1};  \label{Hypothe 3}
\end{equation}

\item \textit{Moreover, one has : }$\forall z\geqslant 0,\forall t,s\in %
\left[ 0,T\right] $%
\begin{equation}
\left \{ 
\begin{array}{l}
\left \Vert \dfrac{\partial ^{2}}{\partial t^{2}}(A(t+t_{n})-\lambda -z)^{-1}-%
\dfrac{\partial ^{2}}{\partial s^{2}}(A(s+s_{n})-\lambda -z)^{-1}\right \Vert
_{L(E)}\leqslant \dfrac{C\left \vert t-s\right \vert ^{2\theta }}{z+1}, \\ 
\\ 
\left \Vert \dfrac{\partial ^{2}}{\partial t^{2}}(A(t+t_{n})-\lambda -z)^{-1}-%
\dfrac{\partial ^{2}}{\partial s^{2}}(A(s+s_{n})-\lambda -z)^{-1}\right \Vert
_{L(E)}\leqslant \dfrac{C\left \vert t-s\right \vert ^{2\theta }}{z+1}.%
\end{array}%
\right.   \label{Hypothese4}
\end{equation}
\end{enumerate}

\begin{rem}
In the sequel the symbol $C$ stands for a generic positive constant except
when other dependence is stated explicitly. On the other hand, it is
important to note that all the constants given above are independent of $t$
and and consequently of $n$.
\end{rem}

\section{Some Regularity Results}

We are concerned with a study of the problem 
\begin{equation}
\left \{ 
\begin{array}{l}
w_{n}^{\prime \prime }(t)+A_{n}(t)w_{n}(t)-\lambda w_{n}(t)=g_{n}(t),\text{\ 
}0\leq t\leq T, \\ 
w_{n}(0)+w_{n}(T)=0, \\ 
w_{n}^{\prime }(0)+w_{n}^{\prime }(T)=0,%
\end{array}%
\right.  \label{EquVariable}
\end{equation}%
where :

\begin{itemize}
\item $g_{n}\in C^{2\theta }\left( \left[ 0,T\right] ,E\right) $, $2\theta
\in ]0,1[$,

\item $A_{n}(t):=A(t+t_{n}).$
\end{itemize}

Our purpose is to establish some results about the existence, uniqueness and
maximal regularity of a strict solution $w_{n}$ for Problem (\ref%
{EquVariable}) by building explicitly a representation of the solution $%
w_{n}(t)$ and studying its optimal regularity. Recall here that a strict
solution is a function $w_{n}$ such that%
\begin{equation*}
\left \{ 
\begin{array}{l}
w_{n}\in C^{2}\left( \left[ 0,T\right] ,E\right) , \\ 
w_{n}\left( t\right) \in D\left( A\left( t\right) \right) \text{ for every }%
t\in \left[ 0,T\right] , \\ 
t\mapsto \left( A_{n}(t)-\lambda \right) w_{n}(t)\in C\left( \left[ 0,T%
\right] ,E\right) ,%
\end{array}%
\right.
\end{equation*}%
and satisfying the anti-periodic boundary conditions%
\begin{equation*}
\begin{array}{l}
w_{n}(0)+w_{n}(T)=0, \\ 
w_{n}^{\prime }(0)+w_{n}^{\prime }(T)=0.%
\end{array}%
\end{equation*}%
The techniques used here are essentially based on the Dunford functional
calculus and the methods applied in \cite{Aquistapace-Terreni}, \cite%
{LabbasTanabe1}-\cite{Labbas} and \cite{Yagi}. We know that if $A_{n}(t)$
is a constant operator satisfying (\ref{Hypoth1}), the representation of the
solution $w_{n}$ is given by the formula%
\begin{equation}
w_{n}(t)=-\frac{1}{2i\pi }\int_{\gamma }\int_{0}^{T}K_{\sqrt{-z}}\left(
t,s\right) (A_{n}-\lambda -z)^{-1}g_{n}(s)dsdz  \label{ReprConstant}
\end{equation}%
where 
\begin{equation}
K_{\sqrt{-z}}\left( t,s\right) =\left \{ 
\begin{array}{ll}
\dfrac{e^{-\sqrt{-z}\left( t-s\right) }-e^{-\sqrt{-z}\left( T-t+s\right) }}{2%
\sqrt{-z}\left( 1+e^{-T\sqrt{-z}}\right) }, & 0\leq s\leq t, \\ 
\dfrac{e^{-\sqrt{-z}\left( s-t\right) }-e^{-\sqrt{-z}\left( T+t-s\right) }}{2%
\sqrt{-z}\left( 1+e^{-T\sqrt{-z}}\right) }, & t\leq s\leq T,%
\end{array}%
\right.  \label{greenkernel}
\end{equation}%
and the curve $\gamma $ is the retrograde oriented boundary of the sector $%
\Pi _{\theta _{0},r_{0}}$ of the form 
\begin{equation}
\Pi _{\delta _{0},r_{0}}=\left \{ z\in 
\mathbb{C}
-\left \{ 0\right \} :\left \vert \arg (z)\right \vert \leqslant \delta
_{0}\right \} \cup \left \{ z\in 
\mathbb{C}
:\left \vert z\right \vert \leqslant r_{0}\right \} ,  \label{secteur}
\end{equation}%
with some small $\delta _{0}>0$, and $r_{0}>0$. (Here, $\rho (A_{n})$ is the
resolvent set of $A_{n}$).

First, it is necessary to note here that

\begin{lem}
\label{LemmeMinDenomP+}There exists $C\left( \delta _{0}\right) >0$ such
that for all $z\in \Pi _{\delta _{0},r_{0}}$, one has :%
\begin{equation*}
\left \vert 1+e^{-T\sqrt{-z}}\right \vert \geqslant C\left( \delta
_{0}\right) .
\end{equation*}
\end{lem}

\begin{proof}
See Lemma 3 in \cite{LMS}.
\end{proof}

\begin{rem}
By using a classical argument of analytic continuation on the resolvent,
the previous assumptions hold true in the sector $\Pi _{\delta
_{0},r_{0}}$, and then, on $\gamma $. Furthermore, we can replace $z$ by $%
z+\lambda .$
\end{rem}

Keeping in mind the constant case (see formula (\ref{ReprConstant})), we
look for a solution of Problem (\ref{EquVariable}) in the following form :%
\begin{equation}
w_{n}(t)=-\frac{1}{2i\pi }\int_{\gamma }\int_{0}^{T}K_{\sqrt{-z}}\left(
t,s\right) (A_{n}\left( t\right) -\lambda -z)^{-1}g_{n}^{\ast }(s)dsdz,
\label{ReprVari}
\end{equation}%
where $g_{n}^{\ast }$ is an unknown function to be determined in some
adequate space in order to obtain a strict solution $w_{n}$ of Problem (\ref%
{EquVariable}), when $g_{n}\in C^{2\theta }\left( \left[ 0,T\right]
,E\right) $.

Our first result concerning the vector valued function $w_{n}(t)$ given by (%
\ref{ReprVari}) is

\begin{prop}
Suppose that $g_{n}^{\ast }\in C^{2\theta }(\left[ 0,T\right] ,E)$, $%
0<2\theta <1$. Then, for all $t\in \left[ 0,T\right] :$%
\begin{equation*}
w_{n}\in C^{2}(\left[ 0,T\right] ,E),
\end{equation*}%
and%
\begin{equation*}
w_{n}(t)\in D\left( A_{n}\left( t\right) \right) .
\end{equation*}
\end{prop}

\begin{proof}
First, observe that the vector valued function $w_{n}$ is well defined. In
fact, using a direct computation on the kernel (\ref{greenkernel}), one has 
\begin{eqnarray*}
&&\left \Vert \int_{0}^{T}K_{\sqrt{-z}}(t,s)(A_{n}(t)-\lambda
-z)^{-1}g_{n}^{\ast }(s)ds\right \Vert  \\
&\leqslant &\left( \underset{t\in \lbrack 0,T]}{\sup }\int_{0}^{T}\left \vert
K_{\sqrt{-z}}(t,s)\right \vert \left \Vert (A_{n}(t)-\lambda
-z)^{-1}\right \Vert _{L\left( E\right) }ds\right) \left \Vert g_{n}^{\ast
}\right \Vert _{C(\left[ 0,T\right] ;E)} \\
&\leqslant &\frac{C}{\left \vert z\right \vert }\left \Vert g_{n}^{\ast
}\right \Vert _{C(\left[ 0,T\right] ;E)}.
\end{eqnarray*}%
Now, we write $w_{n}$ as follows :%
\begin{eqnarray*}
w_{n}(t) &=&-\frac{1}{2i\pi }\int_{\gamma }\int_{0}^{T}K_{\sqrt{-z}%
}(t,s)(A_{n}(t)-\lambda -z)^{-1}\left( g_{n}^{\ast }(s)-g_{n}^{\ast
}(t)\right) dsdz \\
&&-\frac{1}{2i\pi }\int_{\gamma }\int_{0}^{T}K_{\sqrt{-z}}(t,s)(A_{n}(t)-%
\lambda -z)^{-1}g_{n}^{\ast }(t)dsdz,
\end{eqnarray*}%
which becomes :%
\begin{eqnarray*}
w_{n}(t) &=&-\frac{1}{2i\pi }\int_{\gamma }\int_{0}^{T}K_{\sqrt{-z}%
}(t,s)(A_{n}(t)-\lambda -z)^{-1}\left( g_{n}^{\ast }(s)-g_{n}^{\ast
}(t)\right) dsdz \\
&&-\frac{1}{2i\pi }\int_{\gamma }c_{\sqrt{-z}}(t)\frac{(A_{n}(t)-\lambda
-z)^{-1}}{z}g_{n}^{\ast }(t)dz \\
&&-\frac{1}{2i\pi }\int_{\gamma }\frac{(A_{n}(t)-\lambda -z)^{-1}}{z}%
g_{n}^{\ast }(t)dz,
\end{eqnarray*}%
where 
\begin{equation*}
c_{\sqrt{-z}}(t)=\frac{e^{-\sqrt{-z}t}+e^{-\sqrt{-z}\left( T-t\right) }}{%
\left( 1+e^{-T\sqrt{-z}}\right) }.
\end{equation*}%
Thanks to (\ref{Hypoth1}), we have%
\begin{align*}
& \left \Vert \frac{1}{2i\pi }\int_{\gamma }\int_{0}^{T}K_{\sqrt{-z}%
}(t,s)(A_{n}(t)-\lambda -z)^{-1}\left( g_{n}^{\ast }(s)-g_{n}^{\ast
}(t)\right) dsdz\right \Vert  \\
& \leqslant \int_{\gamma }\frac{C}{\left \vert z\right \vert }%
\int_{0}^{T}\left \vert K_{\sqrt{-z}}(t,s)\right \vert \left \vert
t-s\right \vert ^{2\theta }\left \Vert g_{n}^{\ast }\right \Vert _{C^{2\theta }(%
\left[ 0,T\right] ;E)}ds\left \vert dz\right \vert  \\
& \leqslant C\int_{\gamma }\frac{\left \vert dz\right \vert }{\left \vert
z\right \vert ^{2+\theta }}\left \Vert g_{n}^{\ast }\right \Vert _{C^{2\theta }(%
\left[ 0,T\right] ;E)} \\
& \leqslant C\left \Vert g_{n}^{\ast }\right \Vert _{C^{2\theta }(\left[ 0,T%
\right] ;E)}.
\end{align*}%
Concerning the second integral, we have
\begin{equation*}
\left \Vert \frac{1}{2i\pi }\int_{\gamma }c_{\sqrt{-z}}(t)\frac{%
(A_{n}(t)-\lambda )^{-1}(A_{n}(t)-\lambda -z)^{-1}}{z}g_{n}^{\ast
}(t)dz\right \Vert \leqslant C\left \Vert g_{n}^{\ast }\right \Vert _{C(\left[
0,T\right] ;E)}.
\end{equation*}%
On the other hand, by Cauchy Theorem, we deduce that%
\begin{equation*}
\frac{1}{2i\pi }\int_{\gamma }\frac{(A_{n}(t)-\lambda -z)^{-1}}{z}%
g_{n}^{\ast }(t)dz=-(A_{n}(t)-\lambda )^{-1}g_{n}^{\ast }(t).
\end{equation*}%
Summing up, we deduce that 
\begin{equation*}
\forall t\in \left[ 0,T\right] ,\text{ }w_{n}(t)\in D(A_{n}(t))
\end{equation*}%
and 
\begin{align*}
& (A_{n}(t)-\lambda )^{-1}w_{n}(t) \\
& =-\frac{1}{2i\pi }\int_{\gamma }\int_{0}^{\delta }K_{\sqrt{-z}%
}(t,s)(A_{n}(t)-\lambda )^{-1}(A_{n}(t)-\lambda -z)^{-1}\left( g_{n}^{\ast
}(s)-g_{n}^{\ast }(t)\right) dsdz \\
& -\frac{1}{2i\pi }\int_{\gamma }c_{\sqrt{-z}}(t)\frac{(A_{n}(t)-\lambda
)^{-1}(A_{n}(t)-\lambda -z)^{-1}}{z}g_{n}^{\ast }(t)dz \\
& -g_{n}^{\ast }(t).
\end{align*}
\end{proof}

\begin{prop}
Suppose that $g_{n}^{\ast }\in C^{2\theta }(\left[ 0,T\right] ,E)$, $%
0<2\theta <1$. Then, the abstract equation%
\begin{equation*}
w_{n}^{\prime \prime }(t)+A_{n}\left( t\right) w_{n}(t)-\lambda
w_{n}(t)=g_{n}^{\ast }(t)-R_{\lambda }(g_{n}^{\ast })(t)
\end{equation*}%
is satisfied, where%
\begin{eqnarray*}
R_{\lambda }(g_{n}^{\ast })(t) &=&+\frac{1}{i\pi }\int_{\gamma }\int_{0}^{T}%
\frac{\partial }{\partial t}K_{\sqrt{-z}}(t,s)\frac{\partial }{\partial t}%
(A_{n}\left( t\right) -\lambda -z)^{-1}g_{n}^{\ast }(s)dsdz \\
&&-\frac{1}{2i\pi }\int_{\gamma }\int_{0}^{T}K_{\sqrt{-z}}(t,s)\frac{%
\partial ^{2}}{\partial t^{2}}(A_{n}\left( t\right) -\lambda
-z)^{-1}g_{n}^{\ast }(s)dsdz.
\end{eqnarray*}
\end{prop}

\begin{proof}
\textbf{Step 1. }First, regarding the derivative $w_{n}^{\prime }(t)$, we
have :%
\begin{eqnarray*}
w_{n}^{\prime }(t) &=&\frac{1}{2i\pi }\int_{\gamma }\int_{0}^{t}\frac{e^{-%
\sqrt{-z}\left( t-s\right) }+e^{-\sqrt{-z}\left( T-t+s\right) }}{2\left(
1+e^{-T\sqrt{-z}}\right) }(A_{n}(t)-\lambda -z)^{-1}g_{n}^{\ast }(s)dsdz \\
&&-\frac{1}{2i\pi }\int_{\gamma }\int_{t}^{T}\frac{e^{-\sqrt{-z}\left(
s-t\right) }+e^{-\sqrt{-z}\left( T+t-s\right) }}{2\left( 1+e^{-T\sqrt{-z}%
}\right) }(A_{n}(t)-\lambda -z)^{-1}g_{n}^{\ast }(s)dsdz \\
&&-\frac{1}{2i\pi }\int_{\gamma }\int_{0}^{t}\frac{e^{-\sqrt{-z}\left(
t-s\right) }-e^{-\sqrt{-z}\left( T-t+s\right) }}{2\sqrt{-z}\left( 1+e^{-T%
\sqrt{-z}}\right) }\frac{\partial }{\partial t}(A_{n}(t)-\lambda
-z)^{-1}g_{n}^{\ast }(s)dsdz \\
&&-\frac{1}{2i\pi }\int_{\gamma }\int_{t}^{T}\frac{e^{-\sqrt{-z}\left(
s-t\right) }-e^{-\sqrt{-z}\left( T+t-s\right) }}{2\sqrt{-z}\left( 1+e^{-T%
\sqrt{-z}}\right) }\frac{\partial }{\partial t}(A_{n}(t)-\lambda
-z)^{-1}g_{n}^{\ast }(s)dsdz.
\end{eqnarray*}

\textbf{Step 2.} Let us study the second derivative $w_{n}^{\prime \prime
}(t)$. We follow the approach used in \cite{Tanabe}. Let $\varepsilon $ be a
very small positive number and $t$ be such that 
\begin{equation*}
0<\varepsilon \leqslant t\leqslant T-\varepsilon <T,
\end{equation*}%
and $w_{n,\varepsilon }^{\prime }$ be the function defined by%
\begin{eqnarray*}
w_{n,\varepsilon }^{\prime }(t) &=&\frac{1}{2i\pi }\int_{\gamma
}\int_{0}^{t-\varepsilon }\frac{e^{-\sqrt{-z}\left( t-s\right) }+e^{-\sqrt{-z%
}\left( T-t+s\right) }}{2\left( 1+e^{-T\sqrt{-z}}\right) }(A_{n}(t)-\lambda
-z)^{-1}g_{n}^{\ast }(s)dsdz \\
&&-\frac{1}{2i\pi }\int_{\gamma }\int_{t+\varepsilon }^{T}\frac{e^{-\sqrt{-z}%
\left( s-t\right) }+e^{-\sqrt{-z}\left( T+t-s\right) }}{2\left( 1+e^{-T\sqrt{%
-z}}\right) }(A_{n}(t)-\lambda -z)^{-1}g_{n}^{\ast }(s)dsdz \\
&&-\frac{1}{2i\pi }\int_{\gamma }\int_{0}^{t-\varepsilon }\frac{e^{-\sqrt{-z}%
\left( t-s\right) }-e^{-\sqrt{-z}\left( T-t+s\right) }}{2\sqrt{-z}\left(
1+e^{-T\sqrt{-z}}\right) }\frac{\partial }{\partial t}(A_{n}(t)-\lambda
-z)^{-1}g_{n}^{\ast }(s)dsdz \\
&&-\frac{1}{2i\pi }\int_{\gamma }\int_{t+\varepsilon }^{T}\frac{e^{-\sqrt{-z}%
\left( s-t\right) }-e^{-\sqrt{-z}\left( T+t-s\right) }}{2\sqrt{-z}\left(
1+e^{-T\sqrt{-z}}\right) }\frac{\partial }{\partial t}(A_{n}(t)-\lambda
-z)^{-1}g_{n}^{\ast }(s)dsdz.
\end{eqnarray*}%
Observe here that all these integrals are absolutely convergent and 
\begin{equation*}
w_{n,\varepsilon }^{\prime }(t)\rightarrow w_{n}^{\prime }(t),
\end{equation*}%
strongly as $\varepsilon \rightarrow 0$. In addition, we have%
\begin{equation*}
w_{n,\varepsilon }^{\prime \prime }(t):=\Pi _{n,\varepsilon }^{1}\left(
t\right) +\Pi _{n,\varepsilon }^{2}\left( t\right) +\Pi _{n,\varepsilon
}^{3}\left( t\right) +\Pi _{n,\varepsilon }^{4}\left( t\right) ,
\end{equation*}%
where%
\begin{eqnarray*}
\Pi _{n,\varepsilon }^{1}\left( t\right)  &=&-\frac{1}{2i\pi }\int_{\gamma
}\int_{0}^{t-\varepsilon }\sqrt{-z}\frac{e^{-\sqrt{-z}\left( t-s\right)
}-e^{-\sqrt{-z}\left( T-t+s\right) }}{2\left( 1+e^{-T\sqrt{-z}}\right) }%
(A_{n}(t)-\lambda -z)^{-1}g_{n}^{\ast }(s)dsdz \\
&&-\frac{1}{2i\pi }\int_{\gamma }\int_{t+\varepsilon }^{T}\sqrt{-z}\frac{e^{-%
\sqrt{-z}\left( s-t\right) }-e^{-\sqrt{-z}\left( T+t-s\right) }}{2\left(
1+e^{-T\sqrt{-z}}\right) }(A_{n}(t)-\lambda -z)^{-1}g_{n}^{\ast }(s)dsdz,
\end{eqnarray*}%
\begin{eqnarray*}
\Pi _{n,\varepsilon }^{2}\left( t\right)  &=&+\frac{1}{2i\pi }\int_{\gamma }%
\frac{e^{-\sqrt{-z}\varepsilon }+e^{-\sqrt{-z}\left( T-\varepsilon \right) }%
}{2\left( 1+e^{-T\sqrt{-z}}\right) }(A_{n}(t)-\lambda -z)^{-1}g_{n}^{\ast
}(t-\varepsilon )dz \\
&&+\frac{1}{2i\pi }\int_{\gamma }\frac{e^{-\sqrt{-z}\varepsilon }+e^{-\sqrt{%
-z}\left( T-\varepsilon \right) }}{2\left( 1+e^{-T\sqrt{-z}}\right) }%
(A_{n}(t)-\lambda -z)^{-1}g_{n}^{\ast }(t+\varepsilon )dz,
\end{eqnarray*}%
\begin{eqnarray*}
\Pi _{n,\varepsilon }^{3}\left( t\right)  &=&-\frac{1}{2i\pi }\int_{\gamma }%
\frac{e^{-\sqrt{-z}\varepsilon }-e^{-\sqrt{-z}\left( T-\varepsilon \right) }%
}{2\sqrt{-z}\left( 1+e^{-T\sqrt{-z}}\right) }\frac{\partial }{\partial t}%
(A_{n}(t)-\lambda -z)^{-1}g_{n}^{\ast }(t-\varepsilon )dsdz \\
&&+\frac{1}{2i\pi }\int_{\gamma }\frac{e^{-\sqrt{-z}\varepsilon }-e^{-\sqrt{%
-z}\left( T-\varepsilon \right) }}{2\sqrt{-z}\left( 1+e^{-T\sqrt{-z}}\right) 
}\frac{\partial }{\partial t}(A_{n}(t)-\lambda -z)^{-1}g_{n}^{\ast
}(t+\varepsilon )dz,
\end{eqnarray*}%
\begin{eqnarray*}
\Pi _{n,\varepsilon }^{4}\left( t\right)  &=&\frac{1}{2i\pi }\int_{\gamma
}\int_{0}^{t-\varepsilon }\frac{e^{-\sqrt{-z}\left( t-s\right) }+e^{-\sqrt{-z%
}\left( T-t+s\right) }}{2\left( 1+e^{-T\sqrt{-z}}\right) }\frac{\partial }{%
\partial t}(A_{n}(t)-\lambda -z)^{-1}g_{n}^{\ast }(s)dsdz \\
&&-\frac{1}{2i\pi }\int_{\gamma }\int_{t+\varepsilon }^{T}\frac{e^{-\sqrt{-z}%
\left( s-t\right) }+e^{-\sqrt{-z}\left( T+t-s\right) }}{2\left( 1+e^{-T\sqrt{%
-z}}\right) }\frac{\partial }{\partial t}(A_{n}(t)-\lambda
-z)^{-1}g_{n}^{\ast }(s)dsdz \\
&&+\frac{1}{2i\pi }\int_{\gamma }\int_{0}^{t-\varepsilon }\frac{e^{-\sqrt{-z}%
\left( t-s\right) }+e^{-\sqrt{-z}\left( T-t+s\right) }}{2\left( 1+e^{-T\sqrt{%
-z}}\right) }\frac{\partial }{\partial t}(A_{n}(t)-\lambda
-z)^{-1}g_{n}^{\ast }(s)dsdz \\
&&-\frac{1}{2i\pi }\int_{\gamma }\int_{t+\varepsilon }^{T}\frac{e^{-\sqrt{-z}%
\left( s-t\right) }+e^{-\sqrt{-z}\left( T+t-s\right) }}{2\left( 1+e^{-T\sqrt{%
-z}}\right) }\frac{\partial }{\partial t}(A_{n}(t)-\lambda
-z)^{-1}g_{n}^{\ast }(s)dsdz,
\end{eqnarray*}%
\begin{align*}
& \Pi _{n,\varepsilon }^{5}\left( t\right) \\ =&-\frac{1}{2i\pi }\int_{\gamma
}\int_{0}^{t-\varepsilon }\frac{e^{-\sqrt{-z}\left( t-s\right) }-e^{-\sqrt{-z%
}\left( T-t+s\right) }}{2\sqrt{-z}\left( 1+e^{-T\sqrt{-z}}\right) }\frac{%
\partial ^{2}}{\partial t^{2}}(A_{n}(t)-\lambda -z)^{-1}g_{n}^{\ast }(s)dsdz
\\
&-\frac{1}{2i\pi }\int_{\gamma }\int_{t+\varepsilon }^{T}\frac{e^{-\sqrt{-z}%
\left( s-t\right) }-e^{-\sqrt{-z}\left( T+t-s\right) }}{2\sqrt{-z}\left(
1+e^{-T\sqrt{-z}}\right) }\frac{\partial ^{2}}{\partial t^{2}}%
(A_{n}(t)-\lambda -z)^{-1}g_{n}^{\ast }(s)dsdz.
\end{align*}%
Taking into account all properties (\ref{Hypoth1}) to (\ref{Hypothese4}), it
is easy to see all these integrals are absolutely convergent and can be
treated similarly using the Lebesgue dominated convergence theorem. Then, we
obtain the strong convergence 
\begin{equation*}
w_{n,\varepsilon }^{\prime }(t)\rightarrow w_{n}^{\prime }(t)\text{ \ and }%
w_{n,\varepsilon }^{\prime \prime }(t)\rightarrow -(A_{n}(t)-\lambda
-z)^{-1}w_{n}(t)+R_{\lambda }(g_{n}^{\ast })(t)+g_{n}^{\ast }(t),
\end{equation*}%
as $\varepsilon \rightarrow 0$. Hence%
\begin{equation*}
w_{n}^{\prime \prime }(t)=-(A_{n}(t)-\lambda -z)^{-1}w_{n}(t)+R_{\lambda
}(g_{n}^{\ast })(t)+g_{n}^{\ast }(t),
\end{equation*}%
or%
\begin{equation*}
w_{n}^{\prime \prime }(t)+(A_{n}(t)-\lambda -z)^{-1}w_{n}(t)=g_{n}^{\ast
}(t)+R_{\lambda }(g_{n}^{\ast })(t).
\end{equation*}
\end{proof}

The relationship between the vectorial functions $g_{n}$ and $g_{n}^{\ast }$
is given by the following

\begin{prop}
Suppose that $g_{n}^{\ast }\in L^{\infty }(\left[ 0,T\right] ;E)$. Then,
there exists $\lambda ^{\ast }>0$\ such that for all $\lambda \geq \lambda
^{\ast }$, the equation%
\begin{equation*}
g_{n}\left( t\right) =g_{n}^{\ast }(t)-R_{\lambda }(g_{n}^{\ast })(t),
\end{equation*}%
admits a unique solution 
\begin{equation*}
g_{n}^{\ast }\in L^{\infty }(\left[ 0,T\right] ;E).
\end{equation*}
\end{prop}

\begin{proof}
Recall that 
\begin{eqnarray*}
R_{\lambda }(g_{n}^{\ast })(t) &=&+\frac{1}{i\pi }\int_{\gamma }\int_{0}^{T}%
\frac{\partial }{\partial t}K_{\sqrt{-z}}(t,s)\frac{\partial }{\partial t}%
(A_{n}\left( t\right) -\lambda -z)^{-1}g_{n}^{\ast }(s)dsdz \\
&&-\frac{1}{2i\pi }\int_{\gamma }\int_{0}^{T}K_{\sqrt{-z}}(t,s)\frac{%
\partial ^{2}}{\partial t^{2}}(A_{n}\left( t\right) -\lambda
-z)^{-1}g_{n}^{\ast }(s)dsdz.
\end{eqnarray*}%
Then, thanks to (\ref{Hypoth2})-(\ref{Hypothese4}), we see that 
\begin{align*}
& \left \Vert \int_{\gamma }\int_{0}^{T}K_{\sqrt{-z}}(t,s)\frac{\partial ^{2}%
}{\partial t^{2}}(A_{n}\left( t\right) -\lambda -z)^{-1}g_{n}^{\ast
}(s)dsdz\right \Vert \\
& \leqslant C\int_{\gamma }\frac{1}{\left \vert z\right \vert ^{^{3/2}}\left
\vert z+\lambda \right \vert ^{1/2}}\left \vert dz\right \vert \left \Vert
g_{n}^{\ast }\right \Vert _{C(E)}\leqslant \left( C/\lambda ^{\alpha
}\right) \left \Vert g_{n}^{\ast }\right \Vert _{_{L^{\infty }([0,T;E))}},
\end{align*}%
and%
\begin{align*}
& \left \Vert \frac{1}{2i\pi }\int_{\gamma }\int_{0}^{T}\dfrac{\partial }{%
\partial t}K_{\sqrt{-z}}(t,s)\frac{\partial }{\partial t}(A_{n}\left(
t\right) -\lambda -z)^{-1}g_{n}^{\ast }(s)dsdz\right \Vert \\
& \leqslant C\int_{\gamma }\frac{1}{\left \vert z\right \vert ^{^{1/2}}\left
\vert z+\lambda \right \vert ^{^{1/2+1/2}}}\left \vert dz\right \vert \left
\Vert g_{n}^{\ast }\right \Vert _{C(E)}\leqslant \left( C/\lambda
^{1/2+1/2}\right) \left \Vert g_{n}^{\ast }\right \Vert _{_{L^{\infty
}([0,T;E))}}.
\end{align*}%
This implies 
\begin{equation}
\left \Vert R_{\lambda }\right \Vert _{L(L^{\infty }([0,T];E))}=C/\lambda .
\label{estimation on R}
\end{equation}%
Now, to establish the result it suffices to choose $\lambda ^{\ast }>0$ such
that for $\lambda \geq \lambda ^{\ast }$ 
\begin{equation*}
\left \Vert R_{\lambda }\right \Vert _{L(L^{\infty }([0,T];E))}<1.
\end{equation*}
\end{proof}

The following proposition is concerned with the regularity of the operator $%
R_{\lambda }$ needed in order to study the optimal regularity of the
solution we are looking for:

\begin{prop}
Let $g_{n}\in C^{2\theta }([0,T];E).$ Then, there exists $\lambda ^{\ast }>0$
such that for all $\lambda \geqslant \lambda ^{\ast }$,%
\begin{equation*}
R_{\lambda }(g_{n}^{\ast })\left( t\right) \in C^{2\theta }([0,T];E).
\end{equation*}
\end{prop}

\begin{proof}
It suffices to adapt the same techniques delivered in the proof \ of
Proposition 4.3 in \cite{Labbas}.
\end{proof}

This justifies the following result:

\begin{thm}
Let $g_{n}\in C^{2\theta }([0,T];E)$. Then, there exist $\lambda ^{\ast }>0$
such that for all $\lambda \geqslant \lambda ^{\ast }$, the function $w_{n}$
given in the representation (\ref{ReprVari}) is the unique strict solution
of Problem (\ref{EquVariable}) satisfying 
\begin{equation*}
w_{n}\left( .\right) ,\text{ }(A_{n}\left( .\right) -\lambda
)^{-1}w_{n}(.)\in C^{2\theta }([0,T];E).
\end{equation*}
\end{thm}

As a consequence, we have

\begin{cor}
Let $g_{n}^{\ast }\in L^{\infty }([0,T];E)$. Then there exists $\lambda
^{\ast }>0$ and $C>0$ such that for all $\lambda \geqslant \lambda ^{\ast }$%
, the strict solution $w_{n}$ given by (\ref{ReprVari}) fulfills the
estimate 
\begin{equation}
\max \limits_{t}\left \Vert w_{n}(t)\right \Vert _{E}\leq C
\label{important estimate}
\end{equation}
\end{cor}

\begin{proof}[Sketch of the proof]
The calculus are very cumbersome, we just give the main line of the
demonstration. We have
\begin{equation*}
w_{n}^{\prime \prime }(s)+(A_{n}\left( s\right) -\lambda
)^{-1}w_{n}(s)=g_{n}\left( s\right) ,
\end{equation*}%
then, we may deduce that%
\begin{eqnarray*}
w_{n}(t) &=&-\frac{1}{2i\pi }\int_{\gamma }\int_{0}^{T}K_{\sqrt{-z}}\left(
t,s\right) (A_{n}\left( t\right) -\lambda -z)^{-1}g_{n}(s)dsdz \\
&=&-\frac{1}{2i\pi }\int_{\gamma }\int_{0}^{T}K_{\sqrt{-z}}\left( t,s\right)
(A_{n}\left( t\right) -\lambda -z)^{-1}w_{n}^{\prime \prime }(s)dsdz \\
&&-\frac{1}{2i\pi }\int_{\gamma }\int_{0}^{T}K_{\sqrt{-z}}\left( t,s\right)
(A_{n}\left( t\right) -\lambda -z)^{-1}(A_{n}\left( s\right) -\lambda
)^{-1}w_{n}(s)dsdz.
\end{eqnarray*}%
After using integration by parts, we deduce that 
\begin{eqnarray*}
&&-\frac{1}{2i\pi }\int_{\gamma }\int_{0}^{T}K_{\sqrt{-z}}\left( t,s\right)
(A_{n}\left( t\right) -\lambda -z)^{-1}g_{n}(s)dsdz \\
&=&w_{n}(t)+\frac{1}{i\pi }\int_{\gamma }\int_{0}^{T}\frac{\partial }{%
\partial t}K_{\sqrt{-z}}(t,s)\frac{\partial }{\partial s}(A_{n}\left(
s\right) -\lambda -z)^{-1}w_{n}(s)dsdz \\
&&-\frac{1}{2i\pi }\int_{\gamma }\int_{0}^{T}K_{\sqrt{-z}}(t,s)\frac{%
\partial ^{2}}{\partial s^{2}}(A_{n}\left( s\right) -\lambda
-z)^{-1}w_{n}(s)dsdz,
\end{eqnarray*}%
so that
\begin{equation*}
-\frac{1}{2i\pi }\int_{\gamma }\int_{0}^{T}K_{\sqrt{-z}}\left( t,s\right)
(A_{n}\left( t\right) -\lambda -z)^{-1}g_{n}(s)dsdz=\left( 1+R_{\lambda
}(w_{n})\right) (t).
\end{equation*}%
At this level, it is easy to see that the result is a direct consequence of
the estimate (\ref{estimation on R}).
\end{proof}

\section{Coming back to the singular cylindrical domain}

Coming back to the problem (\ref{approached ADE in Qn}) one obtains, for $t\geqslant t_{n},$ 
\begin{equation*}
v_{n}(t)=w_{n}(t-t_{n}).
\end{equation*}%
Thanks to Proposition \ref{important estimate}, a classical argument allows
us to extract a convergent subsequence%
\begin{equation*}
v_{nj}:=v\left( t_{nj}\right) ,
\end{equation*}%
where%
\begin{equation*}
\lim \limits_{n\rightarrow +\infty }t_{nj}=0.
\end{equation*}%
Then, after a passage to the limit, we deduce the following important result

\begin{thm}
Let $g\in C^{2\theta }([0,T];E)$. Then, there exists $\lambda ^{\ast }>0$
such that for $\lambda \geqslant \lambda ^{\ast }$, the problem 
\begin{equation*}
\begin{array}{ll}
v^{\prime \prime }(t)+A\left( t\right) v(t)-\lambda v(t)=f(t), & t\geq 0, \\ 
v(0)+v(T)=0, &  \\ 
v^{\prime }(0)+v^{\prime }(T)=0 & 
\end{array}%
\end{equation*}%
admits a unique strict solution satisfying 
\begin{equation*}
v\left( .\right) ,\text{ }(A\left( .\right) -\lambda )^{-1}v(.)\in
C^{2\theta }([0,T];E).
\end{equation*}
\end{thm}

Applying all the preceeding abstract results and Lemma \ref{effect of change of variables}, our main results concerning the
transformed problem (\ref{transformed problem}) are formulated as follows:

\begin{thm}
\label{resultat_problemTransf}\textit{Let }$f\in C^{2\theta }(\overline{Q})$%
, $0<2\theta <1$. Then, there exists $\lambda ^{\ast }>0$ such that for $%
\lambda \geqslant \lambda ^{\ast }$, Problem (\ref{transformed problem}) 
\textit{has a unique strict solution }$v\in C^{2}\left( \overline{Q}\right) .
$ Moreover, $v$ satisfies the maximal regularity%
\begin{equation*}
\left \{ 
\begin{array}{l}
\partial _{t}^{2}v\in C^{2\theta }(\overline{Q}), \\ 
\text{and} \\ 
\dfrac{1}{\varphi ^{2}\left( t\right) }\Delta v+\dfrac{\varphi ^{\prime
}\left( t\right) }{\varphi ^{2}\left( t\right) }\left \{ \xi \partial _{\xi
}+\eta \partial _{\eta }\right \} v-\lambda v\in C^{2\theta }(\overline{Q}).%
\end{array}%
\right. 
\end{equation*}
\end{thm}

\begin{thm}
\textit{Let }$h\in C^{2\theta }\left( \left[ 0,T\right] ;C\left( \Omega
\right) \right) ,0<2\theta <1.$ Then, there exists $\lambda ^{\ast
}>0$ such that for $\lambda \geqslant \lambda ^{\ast }$, Problem (\ref%
{transformed problem}) \textit{has a unique strict solution }$v\in
C^{2}\left( \Pi \right) .$ Moreover, $v$ satisfies the maximal regularity%
\begin{equation*}
\left \{ 
\begin{array}{l}
\partial _{t}^{2}v\in C_{w}^{2\theta }(\left[ 0,T\right] ;C(\Omega ))), \\ 
\text{and} \\ 
\Delta v-\lambda v\in C_{w}^{2\theta }(\left[ 0,T\right] ;C(\Omega ))),%
\end{array}%
\right. 
\end{equation*}%
where%
\begin{equation*}
C_{w}^{2\theta }(\left[ 0,T\right] ;C(\Omega ))=\left \{ h\in C^{2\theta }(%
\left[ 0,T\right] ;C(\Omega )):\left( \varphi \left( .\right) \right)
^{2\theta }h\in C^{2\theta }(\left[ 0,T\right] ;C(\Omega ))\right \} .
\end{equation*}
\end{thm}

{\bf Acknowledgments.} The second named author is partially supported by grant 174024 of Ministry
of Science and Technological Development, Republic of Serbia.

\bibliographystyle{amsplain}

\end{document}